\documentclass[10pt,twoside]{article}
\usepackage{}
\usepackage{mathrsfs}
\usepackage{amssymb}
\usepackage{amsmath}
\usepackage[all]{xy}
\usepackage{amsthm}
\usepackage{tikz}
\usetikzlibrary{matrix,arrows,calc,snakes,patterns,decorations.markings}

\numberwithin{equation}{section}

\newtheorem{theorem}{Theorem}[section]
\newtheorem{proposition}[theorem]{Proposition}
\newtheorem{definition}[theorem]{Definition}
\newtheorem{remark}[theorem]{Remark}
\newtheorem{lemma}[theorem]{Lemma}
\newtheorem{example}[theorem]{Example}

\newtheorem{corollary}[theorem]{Corollary}

\newcommand{\edge}{\ar@{-}}

\newcommand{\pf}{\noindent\begin {proof}}
\newcommand{\epf}{\end{proof}}

\newcommand{\Hom}{\mbox{\rm Hom}}

\newcommand{\extdim}{\mbox{\rm ext.dim}}

\def\Ker{\mathop{\rm Ker}\nolimits}

\def\mod{\mathop{\rm mod}\nolimits}

\def\pd{\mathop{\rm pd}\nolimits}

\def\min{\mathop{\rm min}\nolimits}

\def\inf{\mathop{\rm inf}\nolimits}
\def\add{\mathop{\rm add}\nolimits}

\def\gldim{\mathop{\rm gl.dim}\nolimits}

\def\rad{\mathop{{\rm rad}}\nolimits}

\def\tridim{\mathop{\rm tri.dim}\nolimits}

\def\Hom{\mathop{\rm Hom}\nolimits}

\def\lim{\mathop{\underrightarrow{\rm lim}}\nolimits}

\def\End{\mathop{\rm End}\nolimits}

\def\mod{\mathop{\rm mod}\nolimits}

\def\pd{\mathop{\rm pd}\nolimits}

\def\min{\mathop{\rm min}\nolimits}

\def\inf{\mathop{\rm inf}\nolimits}
\def\add{\mathop{\rm add}\nolimits}

\def\gldim{\mathop{\rm gl.dim}\nolimits}

\def\rad{\mathop{{\rm rad}}\nolimits}

\def\Hom{\mathop{\rm Hom}\nolimits}

\def\lim{\mathop{\underrightarrow{\rm lim}}\nolimits}

\def\End{\mathop{\rm End}\nolimits}

\def\repdim{\mathop{\rm rep.dim}\nolimits}
\def\wresoldim{\mathop{\rm w.resol.dim}\nolimits}

\def\lrepdis{\mathop{\rm lrep.dis}\nolimits}

\def\C{\mathop{\rm \mathcal{C}}\nolimits}

\def\V{\mathop{\rm \mathcal{V}}\nolimits}
\def\T{\mathop{\rm \mathcal{T}}\nolimits}
\def\I{\mathop{\rm \mathcal{I}}\nolimits}

\def\LL{\mathop{\rm LL}\nolimits}

\title{ \bf The derived dimensions and representation distances of artin algebras
\footnotetext{
2020 Mathematics Subject Classification: 18G20, 16E10, 18E10.}\\
\footnotetext{
Keywords: derived dimension, Igusa-Todorov algebra, endomorphism algebra, left idealized extension. }
\footnotetext{Email addresses: zhengjunling@cjlu.edu.cn (J. Zheng), yyzhang@zjhu.edu.cn (Y. Zhang)
}
}
\vspace{0.2cm}

\author {  Junling  Zheng, Yingying Zhang\thanks{Corresponding author} \\
{\it \scriptsize  Department of Mathematics, China Jiliang University, Hangzhou, 310018, Zhejiang Province, P. R. China
}\\
{\it \scriptsize  Department of Mathematics, Huzhou University, Huzhou, 313000, Zhejiang Province,
P. R. China }
\\
}

\date{ }

\begin{document}

\baselineskip=16pt


\maketitle

\begin{abstract}
There is a well-known class of algebras called Igusa-Todorov algebras which were introduced in relation to finitistic dimension conjecture. As a generalization of Igusa-Todorov algebras, the new notion of $(m,n)$-Igusa-Todorov algebras provides a wider framework for studying derived dimensions. In this paper, we give methods for constructing $(m,n)$-Igusa-Todorov algebras. As an application, we present for general artin algebras a relationship between
the derived dimension and the representation distance. Moreover, we end this paper to show that the main result can be used to give a better upper bound for the derived dimension for some classes of
algebras. 
\end{abstract}

\pagestyle{myheadings}
\markboth{\rightline {\scriptsize  J. L. Zheng et al.\emph{}}}
         {\leftline{\scriptsize The derived dimensions and representation distances of artin algebras }}


\section{Introduction} 

Rouquier introduced in \cite{rouquier2006representation,rouquier2008dimensions} the
 dimension of a triangulated category
under the idea of Bondal and van den Bergh
in \cite{bondal2003generators}.
Roughly speaking, it is an invariant that measures how quickly the category can be built from one object.
This dimension plays
an important role in representation theory(\cite{ballard2012orlov,
bergh2015gorenstein,
chen2008algebras,
han2009derived,
oppermann2012generating,
rouquier2006representation,
rouquier2008dimensions,
zhang2022tau,
zheng2020upper}).
For example, it can be used to compute the representation dimension of
artin algebras (\cite{oppermann2009lower,rouquier2006representation}).
Similar to the dimension of triangulated categories, the extension dimension of an abelian category
was introduced by Beligiannis in \cite{beligiannis2008some}, also see \cite{dao2014radius}. The size of the extension dimension reflects how far an artin algebra is from a finite
representation type since an artin algebra $\Lambda$ is finite representation type if and only if $\extdim \Lambda=0$(see \cite[Example 1.6(i)]{beligiannis2008some}). For more related results on the extension dimension, we refer to \cite{zhang24, zheng2020extension}.

Let $\Lambda$ be a non-semisimple artin algebra, we denote the category of finitely generated left $\Lambda$-modules by $\Lambda$-$\mod$, the bounded derived category of
$\Lambda$-$\mod$ by $D^{b}(\Lambda\text{-}\mod)$, the Beligiannis dimension of $\Lambda$-$\mod$ by $\extdim \Lambda$ (Definition \ref{extensiondimension}), the Rouquier dimension of $D^{b}(\Lambda\text{-}\mod)$ by $\tridim
D^{b}(\Lambda\text{-}\mod)$ (Definition \ref{def-2.3}), the Loewy length of $\Lambda$ by $\LL(\Lambda)$ ,
the global dimension of $\Lambda$ by $\gldim \Lambda$, the representation dimension of $\Lambda$ by $\repdim \Lambda$
 (see \cite{auslander1999representation}), and the weak resolution dimension (Definition \ref{def-2.2}) of $\Lambda$ by  $\wresoldim \Lambda $(see \cite{iyama2003rejective}).
In \cite{zheng2022thedimension}, the authors gave an upper bound for $\tridim D^{b}(\Lambda\text{-}\mod)$ in terms of $\extdim \Lambda$, i.e., $\tridim D^{b}(\Lambda\text{-}\mod)\leqslant 2\extdim \Lambda+1$. The dimensions mentioned above have the
following relation(see \cite{zheng2020extension})
$$\extdim \Lambda \leqslant \min\{\LL(\Lambda)-1, \gldim \Lambda, \repdim \Lambda-2\}.$$ It follows that $\extdim \Lambda$ is always finite, since $\repdim \Lambda$ is always finite(\cite{iyama2003finiteness}).
The finitistic dimension of $\Lambda$ is defined to be the supremum of the projective dimensions of all finitely generated modules of finite projective dimension. The famous finitistic
dimension conjecture claims that the finitistic dimension of $\Lambda$ is finite.

Using the Igusa-Todorov functions, Xi developed new ideas to understand the finistic dimension conjecture (\cite{xi2004finitistic, xi2008finitistic}). For example, he introduced the notion of left idealized extension (Definition \ref{leftideaext}). In \cite{wei2009finitistic}, Wei introduced the notion of $n$-Igusa-Todorov algebras over which the finitisic dimension conjecture holds. In particular, he combined notions of left idealized extensions and Igusa-Todorov algebras providing many new algebras satisfying the finistic dimension conjecture and showed that the class of 2-Igusa-Todorov algebras is closed under taking the endomorphism algebra of projective modules.

In \cite{zheng2022derivedmnIT}, Zheng introduced the notion of $(m,n)$-Igusa-Todorov algebras (Definition \ref{mnITalgebra}),
which is a generalization of $n$-Igusa-Todorov algebras, and for a given $(m,n)$-Igusa-Todorov algebra $\Lambda$, gave a new upper bound for
$\tridim D^{b}(\Lambda\text{-}\mod)$ in terms of $m$ and $n$. In this paper (Theorem \ref{(m,2)-IT} and Theorem \ref{thm-4.4}), we will describe some constructions of $(m,n)$-Igusa-Todorov algebras. As an application, we give a connection between the derived dimensions (Definition \ref{def-2.3}) and left representation distances (Definition \ref{distance}) of general artin algebras, that is, we have $\tridim D^{b}(\Lambda\text{-}\mod) \leqslant 2\lrepdis(\Lambda)+1$ for any artin algbera $\Lambda$.

Our main results can be stated as follows. The following theorem combines endomorphism algebras and $(m,n)$-Igusa-Todorov algebras, generalizing the result of \cite[Theorem 3.10]{wei2009finitistic}.
\begin{theorem}{\rm (Theorem \ref{(m,2)-IT})}\label{maintheorem1.1}
Let $\Lambda$ be an artin algebra and $\Gamma =\End_{\Lambda}(P)^{op}$ the endomorphism algebra of a projective module in $\Lambda\text{-}\mod$. If $\Lambda$ is an $(m,j)$-Igusa-Todorov algebra for some $j\in\{0,1,2\}$, then $\Gamma $ is also an $(m,j)$-Igusa-Todorov algebra.
\end{theorem}
An interesting question to ask is what happens in more general set up in Theorem \ref{maintheorem1.1}. Namely, do the endomorphism algebras of projective modules over $(m,j)$-IT algebras stay $(m,j)$-IT for $j\geqslant3$? At this time the question remains open and we thank the referee for pointing this out to us.
\begin{theorem}{\rm (Theorem \ref{thm-4.4})}\label{maintheorem1.2}
Let $\Lambda $ be an artin algebra, such that there exists a chain of artin algebras
$\Lambda =\Lambda _{0}\subseteq \Lambda_{1}\subseteq \cdots\subseteq\Lambda_{m}=\Gamma$ with $m\geqslant1$, where $\Lambda_{i+1}$
is a left idealized extension of $\Lambda_{i}$ for $0\leqslant i\leqslant m-1$, and $\Gamma$ is representation-finite.
Then $\Lambda$ is an $(m-1,2)$-Igusa-Todorov algebra.
In particular, $\tridim D^{b}(\Lambda\text{-}\mod ) \leqslant 2m+1$ and $\extdim \Lambda \leqslant m+1.$
\end{theorem}

\section{Preliminaries}
\subsection{The dimension of a triangulated category}
  We recall some notions from
  \cite{oppermann2009lower,rouquier2006representation,rouquier2008dimensions}.
  Let $\T$ be a triangulated category and $\I \subseteq {\rm Ob}\T$.
  Let $\langle \I \rangle_{1}$ be the full subcategory of $\T$
  consisting
  of all direct summands of finite direct sums of shifts of
  objects in $\I$, that is, $\langle \I \rangle_{1}=\add\{X[i]\,|\,X\in \I, \,i\in \mathbb{Z}\}$.
  Given two subclasses $\I_{1}, \I_{2}\subseteq {\rm Ob}\T$,
  we denote by $\I_{1}*\I_{2}$
  the full subcategory of all extensions between them, that is,
  $$\I_{1}*\I_{2}=\{ X\mid  X_{1} \longrightarrow X
  \longrightarrow X_{2}\longrightarrow X_{1}[1]\;
  {\rm with}\; X_{1}\in \I_{1}\; {\rm and}\; X_{2}\in \I_{2}\}.$$
  Write $\I_{1}\diamond\I_{2}:=\langle\I_{1}*\I_{2} \rangle_{1}.$
  Then $(\I_{1}\diamond\I_{2})\diamond\I_{3}=\I_{1}
  \diamond(\I_{2}\diamond\I_{3})$
  for any subclass $\I_{1}, \I_{2}$ and $\I_{3}$
  of $\T$ by the octahedral axiom.
  Write
  \begin{align*}
  \langle \I \rangle_{0}:=0,\;
  \langle \I \rangle_{n+1}:=\langle \I
  \rangle_{n}\diamond\langle \I \rangle_{1}\;{\rm for\; any \;}
  n\geqslant 1.
  \end{align*}


  \begin{definition}\label{def-2.3}
{\rm Given a triangulated category $\T$.
\begin{itemize}
\item[$(1)$] (\cite{ rouquier2006representation}) The dimension of a triangulated category $\T$ is defined to be
$$\tridim {\mathcal{T}}:=\inf\{n\geqslant 0\mid\mathcal{T}=\langle T\rangle_{n+1}\ \text{\rm for some}\ T\in\mathcal{T}\},$$
or $\infty$ if no such a $T$ exists.
\item[$(2)$] (\cite{ oppermann2009lower}) For a subcategory $\C$ of $\T$, the dimension of $\C$ is defined to be
$$\tridim _{\T}\C:=\inf\{n\geqslant 0\mid \C \subseteq\langle T\rangle_{n+1}\ \text{\rm for some}\ T\in\mathcal{T}\},$$
or $\infty$ if no such a $T$ exists.
\end{itemize}}
\end{definition}

In this paper, we are mainly concerned with the dimension of the bounded derived category of an artin algebra $\Lambda$ denoted by $\tridim D^{b}(\Lambda\text{-}\mod)$, which is sometimes called derived dimension.

  \subsection{The extension dimension of a module category}
  Let $\Lambda$ be an artin algebra. All subcategories
   of $\Lambda$-$\mod$ are full, additive and closed under isomorphisms
  and all functors between categories are additive.
  For a subclass $\mathcal{U}$ of $\Lambda$-$\mod$,
   we use $\add \mathcal{U}$ to
  denote the subcategory of $\Lambda$-$\mod$ consisting of
  direct summands of finite direct sums of objects in $\mathcal{U}$.
  Let us recall some notions and basic facts (for example,
   see \cite{beligiannis2008some,zheng2020extension}).
  Let $\mathcal{U}_1,\mathcal{U}_2,\cdots,\mathcal{U}_n$
  be subcategories of $\Lambda$-$\mod$.
  Define
  $$\mathcal{U}_1\bullet \mathcal{U}_2:={\add}\{M\in  \Lambda\text{-}\mod
  \mid {\rm there \;exists \;a \;short \; exact \; sequence \;}
  0\rightarrow U_1\rightarrow  M \rightarrow U_2\rightarrow 0\
  $$$${\rm in}\ \Lambda\text{-}\mod\ {\rm with}\; U_1 \in \mathcal{U}_1 \;{\rm and}\;
  U_2 \in \mathcal{U}_2\}.$$
  For any subcategory $\mathcal{U},\mathcal{V}$ and $\mathcal{W}$ of $\Lambda$-$\mod$,
  by \cite[Proposition 2.2]{dao2014radius} we have
$$(\mathcal{U}\bullet\mathcal{V})\bullet\mathcal{W}=\mathcal{U}\bullet(\mathcal{V}\bullet\mathcal{W}).$$
  Inductively, define
  \begin{align*}
  \mathcal{U}_{1}\bullet  \mathcal{U}_{2}\bullet \dots \bullet\mathcal{U}_{n}:=
  \add \{M\in \mid {\rm there \;exists \;a \;short \; exact \; sequence}\
  0\rightarrow U\rightarrow  M \rightarrow V\rightarrow 0  \\{\rm in}\
   \Lambda\text{-}\mod\ {\rm with}\; U \in \mathcal{U}_{1} \;{\rm and}\;
  V \in  \mathcal{U}_{2}\bullet \dots \bullet\mathcal{U}_{n}\}.
  \end{align*}
  For a subcategory $\mathcal{U}$ of $\Lambda$-$\mod$, set
  $[\mathcal{U}]_{0}=0$, $[\mathcal{U}]_{1}=\add\mathcal{U}$,
  $[\mathcal{U}]_{n}=[\mathcal{U}]_1\bullet [\mathcal{U}]_{n-1}$
   for any $n\geqslant 2$.
If $T\in \Lambda\text{-}\mod$, we write $[T]_{n}$ instead of $[\{T\}]_{n}$.

  Let $X\in\Lambda\text{-}\mod$. Assume that $g: P\longrightarrow X$ is an epimorphism where $P$
  is a projective cover of $X$ in $\Lambda\text{-}\mod$, then we write $\Omega^{1}(X):=\Ker g$. Inductively, for any $n\geqslant 2$,
  we write $\Omega^{n}(X):=\Omega^{1}(\Omega^{n-1}(X))$.
  In particular, we set $\Omega^{0}(X):=X.$

  \begin{definition}\label{extensiondimension}
  {\rm (\cite{beligiannis2008some})
  The extension dimension of $\Lambda\text{-}\mod$ is defined to be
  $$\extdim \Lambda:=\inf\{n\geqslant 0\mid \Lambda\text{-}\mod=[ T]_{n+1}\ {\rm for some}\ T\in\Lambda\text{-}\mod\}.$$
  }
  \end{definition}

\begin{definition}\label{def-2.2}
{\rm (\cite[Definition 4.5(2)]{iyama2003rejective})
The {\bf weak resolution dimension} $\wresoldim \Lambda$ of an artin algebra $\Lambda$ is defined as the minimal
number $n\geqslant 0$ which satisfies the following equivalent conditions.

$(i)$ There exists $M\in\mod\Lambda$ such that, for any $X\in \mod\Lambda$,
there exists an exact sequence
 $$0\longrightarrow M_{n}\longrightarrow M_{n-1}\longrightarrow \cdots \longrightarrow M_{0}\longrightarrow Y\longrightarrow 0 $$
with $M_{i}\in\add M$ and $X\in\add Y$.

$(ii)$ There exists $M\in\mod\Lambda$ such that, for any $X\in\mod\Lambda$,
there exists an exact sequence
$$0\longrightarrow Y\longrightarrow M_{0}\longrightarrow M_{1}\longrightarrow\cdots\longrightarrow M_{n}\longrightarrow 0$$
with $M_{i}\in\add M$ and $X\in \add Y$.
}
\end{definition}

\begin{remark}
{\rm The notions of weak resolution dimension were both introduced by Iyama in \cite[Definition 4.5(2)]{iyama2003rejective} and by Oppermann in \cite[Definition 2.4]{oppermann2009lower}. We should remark that their definitions are different. Here we follow Iyama's weak resolution dimension.
}
\end{remark}

The following result established the relationship between the weak resolution dimension, the dimension of a subcategory of the bounded derived category and the extension dimension.

\begin{lemma}\label{lemma2.3}
For an artin algebra $\Lambda$, we have

$(1)$ {\rm (\cite[corollary 3.6]{zheng2020extension} and \cite[Lemma 2.9]{zhang24})} $\wresoldim \Lambda=\extdim \Lambda;$

$(2)$ {\rm (\cite[Lemma 3.2(2)]{zheng2022thedimension}) }$\tridim_{D^{b}(\Lambda\text{-}\mod)} \Lambda\text{-}\mod\leqslant \extdim \Lambda.$
\end{lemma}

Recall from \cite{auslander1999representation} that the representation dimension of $\Lambda$ is defined as
\begin{equation*}
\repdim\Lambda:=
\begin{cases}
\inf\{\gldim\End_{\Lambda}(M)\mid M\ \text{is a generator-cogenerator for}\ \mod\Lambda\},\ \text{if}\ \Lambda\ \text{is non-semisimple;}\\
1,\;\text{if}\ \Lambda\ \text{is semisimple.}
\end{cases}
\end{equation*}

\begin{lemma}{\rm (\cite[Corollary 3.6]{zheng2020extension})}\label{lemextension2}
  For any non-semisimple artin algebra $\Lambda$, we have
  $$\extdim \Lambda \leqslant \min\{\LL(\Lambda)-1, \gldim \Lambda, \repdim \Lambda-2\}.$$
\end{lemma}

\begin{lemma}\label{ext-rep}{\rm (\cite[Example 1.6)(i)]{beligiannis2008some})}
Let $\Lambda$ be an artin algebra. Then $\Lambda$ is representation finite if and only if $\extdim \Lambda=0.$
\end{lemma}

It is clear that $\tridim$ of the bounded derived categories of an algebra $\Lambda$ and of the endomorphism algebra $\End_{\Lambda}T$ of a given tilting module are equal since $\Lambda$ and $\End_{\Lambda}T$ are derived equivalent and $\tridim$ is invariant under derived equivalent (see \cite[Lemma 3.3]{rouquier2008dimensions}). However, the extension dimensions
of the module categories of an algebra and of the endomorphism algebra of a tilting module can be different. In this sense, the extension dimension seems more refined than the
derived dimension, i.e. $\tridim$ of the derived category. In fact, by \cite[\uppercase\expandafter{\romannumeral8} Proposition 4.4]{assem2006elements} if an algebra is a path algebra of Euclidean type and a tilting module
has both a postprojective and preinjective direct summand, then the tilted algebra is representation-finite and their extension dimensions are always different. For the convenience
of the reader, we give an example here (for more details, see
\cite[\uppercase\expandafter{\romannumeral8} Examples 4.4(a)]{assem2006elements}).
\begin{example}{\rm
This is an example of an algebra and
a tilted algebra with different $\extdim$, but equal $\tridim$. Let $\Lambda$ be the path algebra of the Euclidean quiver $Q$ of type $\widetilde{A_{3}}$:
$$
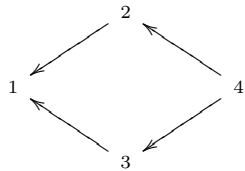
\begin{xy}
(15,0)*+{\begin{smallmatrix} 2\end{smallmatrix}}="1",
(0,-10)*+{\begin{smallmatrix} 1\end{smallmatrix}}="2",
(30,-10)*+{\begin{smallmatrix} 4\end{smallmatrix}}="3",
(15,-20)*+{\begin{smallmatrix} 3\end{smallmatrix}}="4",
\ar"1";"2",\ar"3";"1",\ar"4";"2",\ar"3";"4"
\end{xy}$$
Since $\Lambda$ is representation infinite, we know that $\extdim \Lambda\geqslant 1$ by Lemma \ref{ext-rep}. On the other hand $\gldim \Lambda=1$, hence $\extdim \Lambda\leqslant 1$ by Lemma \ref{lemextension2}. Therefore, $\extdim \Lambda=1$.
Let $T=1\oplus \begin{smallmatrix}4\\3\\1\end{smallmatrix}\oplus \begin{smallmatrix}4\\2\\1\end{smallmatrix}\oplus 4$ be a tilting
module.
Then the tilted algebra $\Gamma =\End_{\Lambda}T$ is representation-finite with $\extdim \Gamma=0$. In fact, $\Gamma$ is given by the quiver
$$
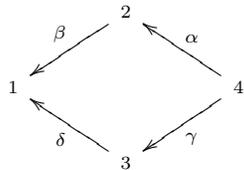
\begin{xy}
(15,0)*+{\begin{smallmatrix} 2\end{smallmatrix}}="1",
(0,-10)*+{\begin{smallmatrix} 1\end{smallmatrix}}="2",
(30,-10)*+{\begin{smallmatrix} 4\end{smallmatrix}}="3",
(15,-20)*+{\begin{smallmatrix} 3\end{smallmatrix}}="4",
\ar_{\beta }"1";"2",\ar_{\alpha }"3";"1",\ar^{\delta }"4";"2",\ar^{\gamma }"3";"4"
\end{xy}$$
bound by $\alpha\beta=0, \gamma\delta=0$.
So we have $\extdim \Lambda\neq \extdim \Gamma$, but we have
$\tridim D^{b}(\Lambda\text{-}\mod)=\tridim D^{b}(\Gamma\text{-}\mod)$ by \cite[Lemma 3.3]{rouquier2008dimensions}
(since $\Lambda$ and $\Gamma$ are derived equivalent).

}
\end{example}
Usually, it is difficult to give the precise value of the extension dimension of module categories. In the following, we deduce the extension dimensions of module categories of some
special algebras.

\begin{example}{\rm
Let $\widetilde{\Lambda}_{l,n}=KQ/I$ with
$$\xymatrix{
{1}\ar@<1.8ex>[r]^{x_{1}}\ar@<-1.5ex>[r]_{x_{n}}^{\vdots}&2\ar@<1.8ex>[r]^{x_{1}}\ar@<-1.5ex>[r]_{x_{n}}^{\vdots}&3&\cdots
&l-1\ar@<1.8ex>[r]^{x_{1}}\ar@<-1.5ex>[r]_{x_{n}}^{\vdots}&l
},$$
$$I=(x_{n^{''}}x_{n^{'}}+x_{n^{'}}x_{n^{''}},x^{2}_{n^{'}}\;|1\leqslant n',n''\leqslant n).$$
See \cite[Example A.7]{oppermann2009lower}, we know that
$$\tridim_{D^{b}(\widetilde{\Lambda}_{l,n}\text{-}\mod)}\widetilde{\Lambda}_{l,n}\text{-}\mod
=\repdim \widetilde{\Lambda}_{l,n}-2=\min\{l-1,n-1\}.$$
By Lemma \ref{lemma2.3}(2)  and Lemma \ref{lemextension2} , we get $$\extdim \widetilde{\Lambda}_{l,n}=\min\{l-1,n-1\}.$$
That is, the extension dimension can be arbitrarily large.
}
\end{example}

\subsection{$(m,n)$-Igusa-Todorov algebras}
Recently, in order to give an upper bound for the derived dimension of an artin algebra $\Lambda$, i.e. $\tridim D^{b}(\Lambda\text{-}\mod)$, Zheng introduced the notion of $(m,n)$-Igusa-Todorov algebras in \cite{zheng2022derivedmnIT} (here, see Remark \ref{ITextensiondimremark}(2)).

  \begin{definition}
  \label{mnITalgebra}
    {\rm (\cite{zheng2022derivedmnIT})
    Let $\Lambda$ be an artin algebra and $m, n$ be non-negative integers. Then $\Lambda$ is said
    to be an $(m,n)$-Igusa-Todorov algebra (or simply $(m,n)$-IT algebra) if there is a module
     $V\in \Lambda\text{-}\mod$
    such that for any module $M$ there exists an exact sequence
    $$0
    \longrightarrow V_{m}
    \longrightarrow V_{m-1}
    \longrightarrow
    \cdots
    \longrightarrow V_{1}
    \longrightarrow V_{0}
    \longrightarrow \Omega^{n}(M)
    \longrightarrow 0
    $$
  where $V_{i} \in \add V$ for each $0 \leqslant  i \leqslant m $.
  Such a module $V$ is said to be an $(m,n)$-IT module.
  }
  \end{definition}
\begin{remark}\label{ITextensiondimremark}
{\rm

(1) By Definition \ref{mnITalgebra} and Lemma \ref{lemma2.3}(1), we know that
the extension dimension of the $(m,n)$-IT algebra is at most $m+n.$

(2) Let $\Lambda$ be an $(m,n)$-IT algebra. Then by \cite[Theorem 1.2]{zheng2022derivedmnIT} $\tridim D^{b}(\Lambda\text{-}\mod) \leqslant 2m+n+1.$

(3) By \cite[Definition 2.1]{wei2009relative} and Definition \ref{mnITalgebra}, we can see that the relative hereditary algebras are $(1,0)$-IT algebras. By Remark \ref{ITextensiondimremark}(2), we see that $\tridim D^{b}(\Lambda\text{-}\mod) \leqslant 3$
for each relative hereditary algebra $\Lambda.$
}
\end{remark}

\section{$(m,n)$-IT algebras and endomorphism algebras}
In this section, we compare an algebra and its endomorphism algebra from the viewpoint of $(m,n)$-IT algebras. Let us begin with the following useful result.
 \begin{lemma}
    \label{syzygy-iso}
    Let $\Lambda$ be an artin algebra and $P$ a projective $\Lambda$-module with $\Gamma =\End_{\Lambda}(P)^{op}$.
     Then we have

{\rm (1)} {\rm (\cite[Lemma 3.1]{huang2013Endomorphism})}For any $X\in \Gamma \text{-}\mod$, there exists a projective $\Lambda$-module $Q$ such that
    $$\Omega_{\Gamma }^{2}(X)\cong \Hom_{\Lambda}(P,\Omega_{\Lambda}^{2}(P\otimes_{\Gamma }X)\oplus Q).$$

{\rm (2)} {\rm(\cite[Lemma 3.5]{wei2009relative})} For any $X\in \Gamma \text{-}\mod$, we have $X\cong \Hom_{\Lambda}(P,P\otimes_{\Gamma }X)$ .

{\rm (3)} For any $X\in \Gamma \text{-}\mod$, there exists a projective $\Lambda$-module $Q$ such that
    $$\Omega_{\Gamma }^{}(X)\cong \Hom_{\Lambda}(P,\Omega_{\Lambda}^{}(P\otimes_{\Gamma }X)\oplus Q).$$
  \end{lemma}
  \begin{proof}
      $(3)$ Consider the following short exact sequence
      $$0 \to \Omega_{\Gamma}(X) \to P_{0} \to X \to 0$$
      where $P_{0}$ is projective in $\Gamma\text{-}\mod $.
      Applying the functor $P\otimes_{\Gamma}-$ to the above short exact sequence, we get the
      following exact sequence
$$P\otimes_{\Gamma}\Omega_{\Gamma}(X) \to P\otimes_{\Gamma}P_{0} \to P\otimes_{\Gamma}X \to 0.$$
Moreover, we have the following short exact sequence in $\Lambda\text{-}\mod$
$$0\to \Omega_{\Lambda}(P\otimes_{\Gamma}X)\oplus Q \to P\otimes_{\Gamma}P_{0} \to P\otimes_{\Gamma}X \to 0,$$
where $Q$ is projective in $\Lambda\text{-}\mod $.
Applying the exact functor $\Hom_{\Lambda}(P,-)$ to the above short eaxct sequence, we can get the following commutative diagram with rows exact
  \[\xymatrix{
0 \ar[r]& \Hom_{\Lambda}(P,\Omega_{\Lambda}(P\otimes_{\Gamma}X)\oplus Q) \ar[r]\ar[d]^{\sigma_{1}}& \Hom_{\Lambda}(P,P\otimes_{\Gamma}P_{0})\ar[r]\ar[d]^{\sigma_{2}}& \Hom_{\Lambda}(P,P\otimes_{\Gamma}X)\ar[r]\ar[d]^{\sigma_{3}}&0\\
0 \ar[r]& \Omega_{\Gamma}(X)\ar[r]&P_{0}\ar[r]  & X\ar[r]&0.
}\]
By Lemma \ref{syzygy-iso}(2), we know that
$\sigma_{2}$ and $\sigma_{3}$ are two isomorphisms. And by Five lemma, we know that $\sigma_{1}$ is also an isomorphism.
That is, $\Omega_{\Gamma }^{}(X)\cong \Hom_{\Lambda}(P,\Omega_{\Lambda}^{}(P\otimes_{\Gamma }X)\oplus Q).$
  \end{proof}
Considering the case of $(m,j)$-IT algebras for $j\in\{0,1,2\}$, we have the following result which shows that the class of all $(m,j)$-IT algebras is closed under taking endomorphism algebras
of projective modules generalizing \cite[Theorem 3.10]{wei2009finitistic}.
  \begin{theorem}\label{(m,2)-IT}
  If $\Lambda$ is $(m,j)$-IT and $\Gamma =\End_{\Lambda}(P)^{op}$ the endomorphism algebra of a projective module in $\Lambda\text{-}\mod$ for some $j\in\{0,1,2\}$, then $\Gamma $ is also an $(m,j)$-IT algebra.
  \end{theorem}

  \begin{proof}
    Let $X\in \Gamma \text{-}\mod$. By Lemma \ref{syzygy-iso}, for $j\in\{0,1,2\}$, we know that
    $\Omega_{\Gamma }^{j}(X)\cong \Hom_{\Lambda}(P,\Omega_{\Lambda}^{j}(P\otimes_{\Gamma }X)\oplus Q)$ with $Q$ a projective $\Lambda$-module.
    Since $\Lambda$ is an $(m,j)$-IT algebra, we know that
    there is a module
       $V\in \Lambda\text{-}\mod$
      such that for any module $M$ there exists an exact sequence
      $$0
      \rightarrow V_{m}
      \rightarrow V_{m-1}
      \rightarrow
      \cdots
      \rightarrow V_{1}
      \rightarrow V_{0}
      \rightarrow \Omega^{j}_{\Lambda}(M)
      \rightarrow 0$$
    where $V_{i} \in \add V$ for each $0 \leqslant  i \leqslant m $.
    In particular, take $M:=P\otimes_{\Gamma }X.$ We have the following exact sequence
    \begin{equation}0
      \rightarrow V_{m}
      \rightarrow V_{m-1}
      \rightarrow
      \cdots
      \rightarrow V_{1}
      \rightarrow V_{0}\oplus Q
      \rightarrow \Omega^{j}_{\Lambda}(P\otimes_{\Gamma }X)\oplus Q
      \rightarrow 0.
      \end{equation}

    Since $P$ is projective, applying the functor $\Hom_{\Lambda}(P,-)$ to (3.1)
    we can get the following exact sequence
    $$0
      \rightarrow \Hom_{\Lambda}(P,V_{m})
      \rightarrow
      \cdots
      \rightarrow \Hom_{\Lambda}(P,V_{1})
      \rightarrow \Hom_{\Lambda}(P,V_{0}\oplus Q)
      \rightarrow \Hom_{\Lambda}(P,\Omega^{j}_{\Lambda}(P\otimes_{\Gamma }X)\oplus Q)
      \rightarrow 0
      $$
    in $\Gamma \text{-}\mod$. Then we get the following exact sequence
    $$0
      \rightarrow \Hom_{\Lambda}(P,V_{m})
      \rightarrow \Hom_{\Lambda}(P,V_{m-1})
      \rightarrow
      \cdots
      \rightarrow \Hom_{\Lambda}(P,V_{1})
      \rightarrow \Hom_{\Lambda}(P,V_{0}\oplus Q)
      \rightarrow \Omega_{\Gamma }^{j}(X)
      \rightarrow 0
      $$
    in $\Gamma \text{-}\mod$, where $\Hom_{\Lambda}(P,V_{i}), \Hom_{\Lambda}(P,V_{0}\oplus Q)\in \add_{\Gamma } (\Hom_{\Lambda}(P,V\oplus \Lambda))$ in $\Gamma \text{-}\mod$
    for each $1\leqslant  i \leqslant m.$
    By Definition \ref{mnITalgebra},
    it follows that  $\Gamma $ is an $(m,j)$-IT algebra.
    \end{proof}

\section{$(m,n)$-IT algebras and left idealized extensions}

In this section we prove that Xi's construction of left idealized extensions creates new
$(m,n)$-IT algebras. First, let us recall:
\begin{definition}\label{leftideaext}
{\rm (\cite{xi2004finitistic})
Let $\Lambda $ and $\Gamma $ be two artin algebras. We say that $\Gamma $ is a left
idealized extension of $\Lambda $, provided that $\Lambda \subseteq \Gamma $ has the same
identity and $\rad \Lambda $ is a left ideal in $\Gamma $.
}
\end{definition}

\begin{remark}{\rm
In \cite[Lemma 4.6]{xi2004finitistic}, Xi showed that
every finite dimensional algebra admits a chain of left
idealized extensions which is of finite length and ends with a representation-finite algebra.
}
\end{remark}
The following result is important in the sequel.
\begin{lemma}\label{7.3}{\rm(\cite[Lemma 0.1 and Lemma 0.2]{xi2005Erratum})}
 Let $\Gamma$ be a left idealized extension of $\Lambda$. Then for any $X\in \Lambda \text{-}{\rm mod}$,
we have $\Omega_{\Lambda}^{2}(X)$ is a left $\Gamma$-module and $\Omega^{2}_{\Lambda}(X)\cong \Omega_{\Gamma}(Z)\oplus Q$ as left $\Gamma$-modules for
some $Z, Q\in \Gamma \text{-}\mod$ with ${_{\Gamma}Q}$ projective.
\end{lemma}

\begin{theorem}\label{thm-4.4}
Let $\Lambda $ be an artin algebra, such that there exists a chain of artin algebras
$\Lambda =\Lambda _{0}\subseteq \Lambda_{1}\subseteq \cdots\subseteq\Lambda_{m}=\Gamma $ with $m\geqslant 1$, where $\Lambda_{i+1}$
is a left idealized extension of $\Lambda_{i}$ for $0\leqslant i\leqslant m-1$, and $\Gamma$ is representation-finite.

$(1)$ Then $\Lambda$ is an $(m-1,2)$-IT algebra.

$(2)$ Furthremore $\tridim D^{b}(\Lambda\text{-}\mod ) \leqslant 2m+1$ and $\extdim \Lambda \leqslant m+1.$
\end{theorem}

\begin{proof}

(1) Let $X_{0}\in\mod \Lambda$. Since $\Lambda_{1}$ is a left idealized extension of $\Lambda$, by Lemma \ref{7.3}, we have
$\Omega^{2}_{\Lambda}(X_{0})\cong \Omega_{\Lambda_{1}}(X_{1})\oplus P_{1}$ in
$\Lambda_{1}\text{-}\mod $ for some $X_{1},P_{1}\in \Lambda_{1}\text{-}\mod$
with $_{\Lambda_{1}}P_{1}$ projective.
Inductively, we can get the isomorphism
$$\Omega^{2}_{\Lambda_{i}}(X_{i})\cong \Omega_{\Lambda_{i+1}}(X_{i+1})\oplus P_{i+1}$$ in
$\Lambda_{i+1}\text{-}\mod $ for some $X_{i+1},P_{i+1}\in \Lambda_{i+1}\text{-}\mod$
with $_{\Lambda_{i+1}}P_{i+1}$ projective for each $0 \leqslant i\leqslant m-1$.
On the other hand, there is an exact sequence
$$0 \longrightarrow \Omega_{\Lambda_{i+1}}^{2}(X_{i+1}) \longrightarrow P'_{i+1}
 \longrightarrow  \Omega_{\Lambda_{i+1}}(X_{i+1})\longrightarrow 0$$
in $\Lambda_{i+1}\text{-}\mod $, where $P'_{i+1}$ is the projective cover of
$\Omega_{\Lambda_{i+1}}(X_{i+1})$.
Then we have the following exact sequence
$$0 \longrightarrow \Omega_{\Lambda_{i+1}}^{2}(X_{i+1}) \longrightarrow P'_{i+1}\oplus P_{i+1}
 \longrightarrow  \Omega^{2}_{\Lambda_{i}}(X_{i})\longrightarrow 0$$
in $\Lambda_{i+1}\text{-}\mod $, for each $0 \leqslant i\leqslant m-1$.
They are also exact in $\Lambda\text{-}\mod.$

Therefore, we have the following exact sequence
$$0 \longrightarrow \Omega_{\Lambda_{m-1}}^{2}(X_{m-1}) \longrightarrow P'_{m-1}\oplus P_{m-1}
\longrightarrow \cdots \longrightarrow  P'_{1}\oplus P_{1} \longrightarrow\Omega^{2}_{\Lambda}(X_{0})  \longrightarrow 0$$
in $\Lambda\text{-}\mod $,
 and $  P'_{i}\oplus P_{i}\in \add_{\Lambda}(\oplus_{i=1}^{m-1} \Lambda_{i})$
for each $1 \leqslant i\leqslant m-1$.

Since $\Gamma=\Lambda_{m}$ is representation finite, then there exists a module $W$
such that $\add W=\mod \Lambda_{m}$.
On the other hand, by Lemma \ref{7.3}, we know that
$\Omega_{\Lambda_{m-1}}^{2}(X_{m-1})\in \mod \Lambda_{m}$,
that is, $\Omega_{\Lambda_{m-1}}^{2}(X_{m-1})\in \add W$.
Now, by Definition \ref{mnITalgebra}, we know that
$\Lambda$ is an $(m-1,2)$-IT algebra.

(2) By Remark \ref{ITextensiondimremark}(2),
 we get that
$\tridim D^{b}(\Lambda\text{-}\mod ) \leqslant 2(m-1)+2+1=2m+1.$ By Remark \ref{ITextensiondimremark}(1),
 we get that
$\extdim \Lambda \leqslant (m-1)+2= m+1.$
\end{proof}
\begin{definition}\label{distance}{\rm (\cite[Page 9050]{xi2005finitistic})
  Given an artin algebra $\Lambda$, 
  the left representation distance of $\Lambda$,
  denoted by $\lrepdis(\Lambda)$, is defined as
  $\lrepdis(\Lambda):=0$ if
  $\Lambda$ is representation finite; and
\begin{center}
$\lrepdis(\Lambda):={\rm min}\{s\,|\,
\text{there exists a chain of subalgebras}\, \Lambda=\Lambda_{0}\subsetneq \Lambda_{1}\subsetneq \Lambda_{2}\subsetneq \cdots \subsetneq \Lambda_{s}\,\text{such that}\,\rad (\Lambda_{i})\,\text{is a left ideal in}\, \Lambda_{i+1}\, \text{for all}\, i \,\text{and that}\, \Lambda_{s}\, \text{is representation finite}\}$
\end{center}
if $\Lambda$ is not representation finite.}
\end{definition}
By Theorem \ref{thm-4.4} and Definition \ref{distance}, we have
\begin{corollary}\label{cor.dim}
For any artin algebra, we have

$(1)$ $\tridim D^{b}(\Lambda\text{-}\mod) \leqslant 2\lrepdis(\Lambda)+1;$

$(2)$ $\extdim \Lambda \leqslant \lrepdis(\Lambda)+1.$
\end{corollary}

The following proposition gives a family of algebras $A$ and their 2-step left idealized extensions $C \subseteq B \subseteq A$ illustrating how Theorem \ref{thm-4.4}, gives upper bounds for $\tridim$ and $\extdim$,
which are much better then the bounds obtained by some other known methods.

\begin{proposition}\label{ex}{\rm
Let $A$ be the algebra (over a field) defined by the following quiver
$$\xymatrix{\bullet
&\bullet\ar[l]_{\lambda}^(1){5}^(0){2}
&\bullet\ar[l]_{\epsilon}^(0){3}
&\bullet\ar[l]_{\xi}^(0){1} &\bullet\ar[l]_{\beta}^(0){4}
&\bullet\ar[l]_{\alpha}^(0){6}
&\bullet\ar[l]_{\sigma_1}^(0){7}
&\bullet\ar[l]_{\sigma_2}^(0){8}
\cdots
&\bullet\ar[l]_{\sigma_{s-6}}^(0){s}
&\bullet\ar[l]_{\tau_1}^(0){s+1}
&\bullet\ar[l]_{\tau_2}^(0){s+2}
\cdots
&\bullet\ar[l]_{\tau_t}^(0){s+t}
}$$
with relations: $\alpha\beta\xi\epsilon\lambda=0,\tau_1\sigma_{s-6}=0,\tau_{i+1}\tau_i=0,1\leqslant i<t$. And assume that $s\geqslant 7$ and $t\textgreater 1$.

Let $B$ be the subalgebra of $A$ generated by $$\{e_{1}, e_{2'}:=e_{2}+e_{4}+e_{5},
e_{3'}=e_{3}+e_{6},e_{7},e_8,\cdots,e_{s+t}, \lambda, \beta,
\alpha,\epsilon, \gamma:=\xi \epsilon, \delta:=\beta \xi,\sigma_{1},\sigma_{2},\cdots,\sigma_{s-6},\tau_1,\tau_2,\cdots,\tau_t \},$$where $e_{i}$ is the primitive idempotent element of $A$ corresponding to the vertex $i$.

Let $C$ be the subalgebra of $B$ generated by
$\{e_{1}, e_{2'},
e_{3'},e_{7},e_8,\cdots,e_{s+t}, \lambda, \beta,
\alpha+\epsilon, \gamma, \delta,\sigma_{1},\sigma_{2},\cdots,\sigma_{s-6},\\ \tau_1,\tau_2,\cdots,\tau_t \}$.
Then:

$(a)$ The chain $C \subseteq B \subseteq A$
 is a chain of idealized extensions with algebra $A$ of finite representation type.

$(b)$ The algebra $C$ is a $(1, 2)$-IT algebra.

$(c)$ $\tridim D^{b}(\mod C)  \leqslant 5$ and $\extdim C\leqslant 3.$

$(d)$ The bounds from known results \cite{rouquier2008dimensions}, \cite{zheng2022derivedmnIT}, \cite{zheng2022thedimension}, \cite{zheng2020extension} are: $\tridim D^{b}(\mod C)  \leqslant \min\{s-3, t+8\}$ and $\extdim C\leqslant t+5.$

}\end{proposition}

\begin{proof}
It is obvious that $A$ is a Nakayama algebra, and hence representation-finite.
We can check that $B$ can be described by the following quiver
$$
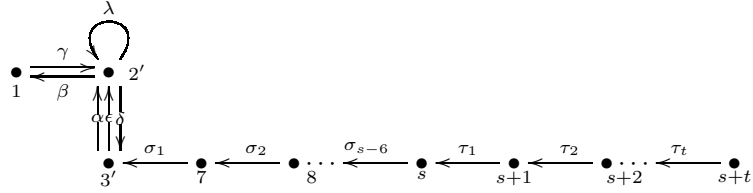
\xymatrix{
\bullet\ar@<0.4ex>[r]^{\gamma}
&\bullet\ar@<0.4ex>[l]^{\beta}^(1){1}
\ar@(ur,ul)_{\lambda}\ar@<1ex>[d]|-{\delta}^(0){2'}
&&&&&&\\
&\bullet\ar@<1ex>[u]|-{\alpha}
\ar[u]|-{\epsilon}
&\bullet\ar[l]_{\sigma_1}^(0){7}^(1){3'}
&\bullet\ar[l]_{\sigma_2}^(0){8}
\cdots
&\bullet\ar[l]_{\sigma_{s-6}}^(0){s}
&\bullet\ar[l]_{\tau_1}^(0){s+1}
&\bullet\ar[l]_{\tau_2}^(0){s+2}
\cdots
&\bullet\ar[l]_{\tau_t}^(0){s+t}
}$$
with relations $\beta\gamma=\delta\epsilon,
\gamma\beta=\gamma\delta=\lambda^{2}=\lambda\beta=\lambda\delta
=\delta\alpha=\epsilon\beta=\epsilon\delta=\alpha\lambda
=\alpha\beta\gamma\lambda=0,\tau_1\sigma_{s-6}=0,\tau_{i+1}\tau_i=0,1\leqslant i<t$.

And $C$ can be described by quiver
$$
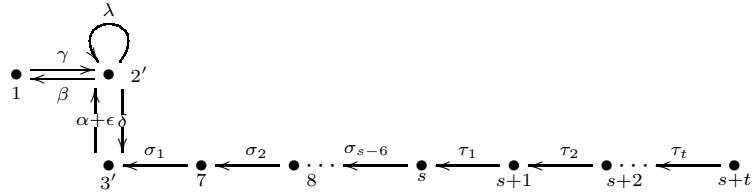
\xymatrix{
\bullet\ar@<0.4ex>[r]^{\gamma}
&\bullet\ar@<0.4ex>[l]^{\beta}^(1){1}
\ar@(ur,ul)_{\lambda}\ar@<1.2ex>[d]|-{\delta}^(0){2'}
&&&&&&\\
&\bullet\ar@<1.2ex>[u]|-{\alpha+\epsilon}
&\bullet\ar[l]_{\sigma_1}^(0){7}^(1){3'}
&\bullet\ar[l]_{\sigma_2}^(0){8}
\cdots
&\bullet\ar[l]_{\sigma_{s-6}}^(0){s}
&\bullet\ar[l]_{\tau_1}^(0){s+1}
&\bullet\ar[l]_{\tau_2}^(0){s+2}
\cdots
&\bullet\ar[l]_{\tau_t}^(0){s+t}
}$$
with relations $\beta\gamma=\delta(\alpha+\epsilon),
\gamma\beta=\gamma\delta=\lambda^{2}=\lambda\beta=\lambda\delta=(\alpha+\epsilon)\beta\gamma\lambda=0,
\tau_1\sigma_{s-6}=0,\tau_{i+1}\tau_i=0,1\leqslant i<t$.

(a) It is easy to check that $C \subseteq B \subseteq A$ is a chain of subalgebras of $A$ such that $\rad(C)$ is a left ideal of $B$ and
$\rad(B)$ is a left ideal of $A$. Therefore the chian $C \subseteq B \subseteq A$ is a chain of idealized extensions with algebra $A$ of finite representation type by Definition \ref{leftideaext}.

(b) By (a) and Theorem \ref{thm-4.4}(1), $C$ is $(1, 2)$-IT.

(c) By (a) and Theorem \ref{thm-4.4}(2), we have $\tridim D^{b}(\mod C)  \leqslant 2\times 2+1=5$ and $\extdim C\leqslant 2+1=3$.

(d) The Loewy structure of the indecomposable projective $C$-modules can be listed as follows.
$${\footnotesize
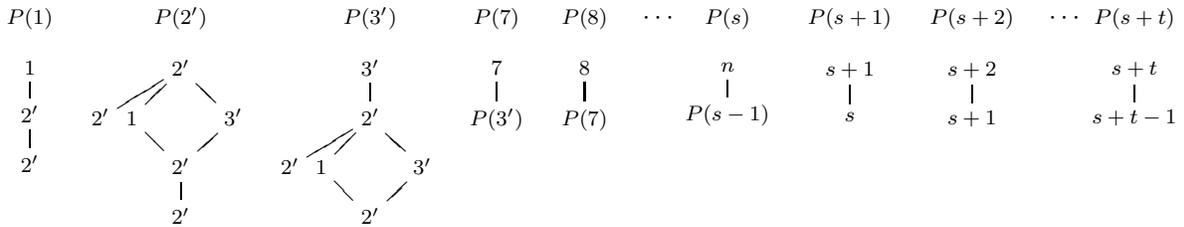
\xymatrix@R=.23cm@C=.01cm{
P(1)\\
1\ar@{-}[d]\\
2'\ar@{-}[d]\\
2'
}\quad
\xymatrix@R=.23cm@C=.01cm{
&&P(2')&\\
&&2'\ar@{-}[dll]\ar@{-}[dl]\ar@{-}[dr]&\\
2'&1\ar@{-}[dr]&&3'\ar@{-}[dl]\\
&&2'\ar@{-}[d]&\\
&&2'&
}\quad
\xymatrix@R=.23cm@C=.01cm{
&&P(3')&\\
&&3'\ar@{-}[d]&\\
&&2'\ar@{-}[dl]\ar@{-}[dr]\ar@{-}[dll]&\\
2'&1\ar@{-}[dr]&&3'\ar@{-}[dl]\\
&&2'&\\
}\quad
\xymatrix@R=.23cm@C=.01cm{
P(7)\\
7\ar@{-}[d]\\
P(3')
}\quad
\xymatrix@R=.23cm@C=.01cm{
P(8)\\
8\ar@{-}[d]\\
P(7)
}\quad\cdots
\xymatrix@R=.23cm@C=.01cm{
P(s)\\
n\ar@{-}[d]\\
P(s-1)
}\quad
\xymatrix@R=.23cm@C=.01cm{
P(s+1)\\
s+1\ar@{-}[d]\\
s
}\quad
\xymatrix@R=.23cm@C=.01cm{\small
P(s+2)\\
s+2\ar@{-}[d]\\
s+1
}\quad\cdots
\xymatrix@R=.23cm@C=.01cm{
P(s+t)\\
s+t\ar@{-}[d]\\
s+t-1
}
}$$
Then $\LL(C)=s-2$. Consider the following projective resolutions
$$0\to S(2')\to  P(2') \to P(1)\oplus P(3')\to P(2')\to P(1)\to S(1)\to 0,$$
$$0 \to M\oplus S(2') \to P(2') \to S(2') \to 0, \mbox{ and}$$
$$0\to S(2')\to P(2')  \to P(3') \to S(3') \to 0,$$
where the Loewy sturcture of $M$ is as follows:
$$\xymatrix@R=.23cm@C=.01cm{
1\ar@{-}[dr]&&3'\ar@{-}[dl]\\
&2'\ar@{-}[d]&\\
&2'&
}$$
Thus $\pd S(1)=\pd S(2')=\pd S(3')=\infty$ and $\gldim C=\infty$. Note that there are short exact sequences
$$0\to P(3')\to P(7)\to S(7)\to 0,$$
$$0\to P(i-1)\to P(i)\to S(i)\to 0, \mbox{ for } 8\leqslant i\leqslant s, \mbox{ and }$$
$$0\to S(s+j-1)\to P(s+j)\to S(s+j)\to 0, \mbox{ for } 1\leqslant j\leqslant  t.$$
Thus $\pd S(i)=1$ for $7\leqslant i\leqslant  s$ and $\pd S(s+j)=j+1$ for $1\leqslant  j\leqslant  t$.

Let $\V=\{S(i)\mid 7\leqslant i\leqslant  s+t\}$. Then $\pd \V=t+1$. By calculation, we have $\ell\ell^{t_{\V}}(C)=4$. By \cite[Proposition 7.4, Proposition 7.37]{rouquier2008dimensions} and \cite[Theorem 1.3]{zheng2022derivedmnIT}, we have
$$\tridim D^{b}(\mod C)  \leqslant \inf \{\gldim C, \LL(C)-1,2\ell\ell^{t_{\V}}(C)+\pd \V-1\}=\min\{s-3, t+8\}.$$
On the other hand, $\extdim C\leqslant \ell\ell^{t_{\V}}(C)+\pd \V =t+5$  by \cite[Corollary 3.15(2.1)]{zheng2022thedimension} or  \cite[Theorem 3.19]{zheng2020extension}.

\end{proof}

\begin{remark}{\rm
Note that by Theorem \ref{thm-4.4} we may get sometimes better upper bounds than \cite{rouquier2008dimensions}, \cite{zheng2022derivedmnIT}, \cite{zheng2022thedimension} and \cite{zheng2020extension}. For example, Proposition \ref{ex}(c) gives better upper bounds than (d) for $s\gg 0$ and $t\gg 0$.
}
\end{remark}

{\bf Acknowledgements.} Junling Zheng was supported by the National Natural Science Foundation of China (Grant No. 12001508). Yingying Zhang was supported by the National Natural Science Foundation of China (Grant No. 12201211). The authors would like to thank Jinbi Zhang for his helpful discussions. Moreover, the authors thank the referee for the detailed suggestions which made the paper more readable.
\vspace{0.2cm}

{\bf Conflict of interest.} The authors declare that they have no conflict of interest.

\vspace{0.6cm}
%


%
\bibliographystyle{abbrv}

\end{document}